\documentclass[graybox]{svmult}

\usepackage{mathptmx}       
\usepackage{helvet}         
\usepackage{courier}        
\usepackage{type1cm}        
\usepackage{makeidx}         
\usepackage{graphicx}        
\usepackage{multicol}        
\usepackage[bottom]{footmisc}

\usepackage{amsmath,amssymb,latexsym}
\usepackage{amscd}
\usepackage{rotating}
\usepackage[mathscr]{eucal}

\makeindex             

\def\w{\tilde}

\def\lim{\operatorname{lim}}

\def\codim{\operatorname{codim}}

\def\id{\operatorname{id}}

\begin{document}

\title*{Connectivity and a Problem of Formal Geometry}
\author{Lucian B\u adescu}

\institute{Lucian B\u adescu \at Universit\`a degli Studi di Genova, Dipartimento di Matematica, Via
Dodecaneso 35, 16146 Genova, Italy, \email{badescu@dima.unige.it}}
%
%
\maketitle

\centerline{\small To Alexandru Dimca and \c Stefan Papadima on the occasion of their sixtieth 
anniversaries}

\vskip2cm

\abstract*{Let $P=\mathbb P^m(e)\times\mathbb P^n(h)$ be a product of weighted projective spaces, and let $\Delta_P$ be the diagonal of $P\times P$. We prove an algebraization result for formal-rational functions on certain closed subvarieties $X$ of  $P\times P$ along the intersection $X\cap\Delta_P$}

\abstract{Let $P=\mathbb P^m(e)\times\mathbb P^n(h)$ be a product of weighted projective spaces, and let $\Delta_P$ be the diagonal of $P\times P$. We prove an algebraization result for formal-rational functions on certain closed subvarieties $X$ of  $P\times P$ along the intersection $X\cap\Delta_P$.}

\section{Introduction}
\label{sec:1}
Let $P$ be a projective irreducible variety and let $f\colon X\to P\times P$ be a morphism from a complete irreducible variety $X$ over an algebraically closed field $k$. Denote by $\Delta_P$ the diagonal of $P\times P$. Then one may ask under which conditions the inverse image $f^{-1}(\Delta_P)$ is connected (resp. non-empty). Here by a connected scheme we shall mean a non-empty scheme whose underlying topological space is connected. The first result in this direction is the famous theorem of Fulton and Hansen (\cite{FH}) which states that the answer to this question is affirmative if  $P=\mathbb P^n$ and if  $\dim f(X)>n$ (resp.  if  $\dim f(X)\geq n$). That result has a lot of interesting geometric applications (see \cite{FL}).

The connectivity result of Fulton and Hansen has been generalized in various directions by Hansen
\cite{Han}, Faltings \cite{F}, \cite{F2}, Debarre \cite{D}, \cite{D2}, \cite{D3}, B\u adescu \cite{B1}, \cite{Bad}, B\u adescu-Repetto \cite{BRe}, and others. 

On the other hand, in \cite{B1} and \cite{Bad} the connectivity results of Fulton-Hansen \cite{FL} and Debarre \cite{D} have been improved to get stronger conclusions involving the $G3$ condition of Hironaka-Matsumura \cite{HM} on the extension of formal-rational functions on $X$ along $f^{-1}(\Delta_P)$ (see Definition \ref{oo} below). The aim of the present paper is to improve the connectivity
result of \cite{BRe} in the same spirit.

To state our main result, let $P$ denote the product $\mathbb P^m(e)\times\mathbb P^n(h)$ of weighted projective spaces $\mathbb P^m(e)$ and $\mathbb P^n(h)$ of weights 
$e=(e_0,\ldots, e_m)$ and $h=(h_0,\ldots,h_n)$ respectively, with $e_i,h_j\geq 1$,  $i=0,\ldots,m$ and $j=0,\ldots,n$. Let $f\colon X\to P\times P$ be a morphism from a complete irreducible variety $X$. Denote by $X_{13}\subseteq\mathbb P^m(e)\times\mathbb P^m(e)$ (resp.  by $X_{24}\subseteq\mathbb P^n(h)\times\mathbb P^n(h)$) the image of $f(X)$ under the projection 
$p_{13}$  of $P\times P=\mathbb P^m(e)\times\mathbb P^n(h)\times\mathbb P^m(e)\times\mathbb P^n(h)$ onto $\mathbb P^m(e)\times\mathbb P^m(e)$
(resp. under the projection $p_{24}$ onto $\mathbb P^n(h)\times\mathbb P^n(h)$).
For the basic properties of weighted projective spaces, see \cite{Do} or \cite{BR}.

Precisely, our aim is to prove the following strengthening of the connectivity result of \cite{BRe} (see Theorem \ref{br} below), and, under a slightly stronger hypothesis, also a generalization of the main result of \cite{Bad}:

\medskip

\noindent{\bf Theorem} (=Theorem \ref{br'}  below) {\em Under the above notation, let  $f\colon X\to P\times P$ be a morphism from a complete irreducible variety $X$, with
$P:=\mathbb P^m(e)\times\mathbb P^n(h)$, the product of the weighted  projective spaces $\mathbb P^m(e)$ and $\mathbb P^n(h)$ over an algebraically closed field $k$. Let $\Delta_P$ be the diagonal of $P\times P$ and set $a:=\max\{m+\dim X_{24},n+\dim X_{13}\}$. If  $\dim f(X)> a$ then $f^{-1}(\Delta_P)$ is $G3$ in $X$, i.e. the canonical injective map $\alpha\colon K(X)\to K(X_{/f^{-1}(\Delta_P)})$, from the field $K(X)$ of rational functions of $X$ to the ring $K(X_{/f^{-1}(\Delta_P)})$ of formal-rational functions of $X$ along $f^{-1}(\Delta_P)$, is an isomorphism $($see the Definition $\ref{oo}$ below$)$.}

\medskip

In other words, this theorem is an extension result of the formal rational functions on $X$ along $X\cap\Delta_P$ to rational functions on $X$. Let me explain why this theorem is an improvement of the connectivity result proved in \cite{BRe}.
By Theorem \ref{(2.1.24)}  below the connectivity result of \cite{BRe} (see Theorem \ref{br} below) is equivalent to saying that the ring $K(f(X)_{/f(X)\cap\Delta_P})$ is a field and the subfield $K(f(X))$ is algebraically closed in  $K(f(X)_{/f(X)\cap\Delta_P})$, while the above theorem is equivalent to saying that the natural map $K(f(X))\to K(f(X)_{/f(X)\cap\Delta_P})$ is an isomorphism.

To prove this result we use an extension theorem for formal-rational functions for the case $P=\mathbb P^m\times\mathbb P^n$ proved in \cite{Bad} (see Theorem \ref{debarre} below) and the connectivity result proved in \cite{BRe} (Theorem \ref{br} below), via some basic known results on formal-rational functions.

Here are two consequences of the above Theorem:

\medskip

{\bf Corollary 1} {\em  Let  $f\colon X\to P\times P$ be a morphism from a complete irreducible variety $X$, with  $P=\mathbb P^m(e)\times \mathbb P^n(h)$ over an algebraically closed field of arbitrary characteristic such that $m\geq n\geq 1$ and $\codim_{P\times P}f(X)< n$. Then $f^{-1}(\Delta)$ is $G3$ in $X$.}

\medskip

In the special case when $P=\mathbb P^m\times\mathbb P^n$ is a product of two ordinary projective spaces over an algebraically closed  field of characteristic zero and $f$ is a closed embedding, Corollary 1  also follows from an old general result of Faltings (see \cite{F}, Satz 8, page 161) proved in the case when $P$ is a complex projective rational homogeneous space. In general $P=\mathbb P^m(e)\times\mathbb P^n(h)$ is  singular, so that Corollary 1 (to our best knowledge) is new.

\medskip

{\bf Corollary 2} {\em  Let $X$ and $Y$ be two closed  irreducible subvarieties of $P=\mathbb P^m(e)\times\mathbb P^n(h)$  such that  $m\geq n\geq 1$ and $\dim X+\dim Y>
2m+n$. Then $X\cap Y$ is $G3$ in $X$ and in $Y$.}

\medskip

Corollary 2 extends to the case $P=\mathbb P^m(e)\times\mathbb P^n(h)$ an old result of Faltings  \cite{F2} proved (by local methods) if $P=\mathbb P^n$.

The paper is organized as follows.
In the first section we recall some known results that will be needed in Section 2. In the second section we prove the theorem and the two corollaries stated above. 

\medskip

{\em Terminology and notation.} Unless otherwise specified, we shall use the standard terminology and notation in algebraic geometry.
We shall work over an algebraically closed ground field $k$ of arbitrary characteristic.

\section{Background material}
\label{sec:2}

In this section we gather together the known results which are going to be used in Sections 2.

\begin{theorem}[B\u adescu-Repetto \cite{BRe}]\label{br} Under the  notation of the introduction, let  $f\colon X\to P\times P$ be a morphism from a complete irreducible variety $X$, with
$P=\mathbb P^m(e)\times\mathbb P^n(h)$ the product of the weighted  projective spaces $\mathbb P^m(e)$ and $\mathbb P^n(h)$ over an algebraically closed field $k$. Let $\Delta_P$ be the diagonal of $P\times P$ and set $a:=\max\{m+\dim X_{24},n+\dim X_{13}\}$. Then the following statements hold true:
\begin{enumerate}\item[\em i)] If  $\dim f(X)\geq a$ then $f^{-1}(\Delta_P)$ is nonempty, and 
  \item[\em ii)] If  $\dim f(X)> a$ then $f^{-1}(\Delta_P)$ is connected.\end{enumerate}\end{theorem}
  
  \begin{remark}\label{n=0}  If in Theorem \ref{br} we take $n=0$ then $X_{24}$ is a point and hence $a=\max\{m,\dim X_{13}\}$. Then $P\cong\mathbb P^m$, $f(X)\cong X_{13}$, and therefore the conclusion  of Theorem \ref{br} becomes: 
 \begin{enumerate}\item[i')] If  $\dim f(X)\geq m$ then $f^{-1}(\Delta_P)\neq\varnothing$, and 
  \item[ii')] If  $\dim f(X)> m$ then $f^{-1}(\Delta_P)$ is connected.\end{enumerate}
In other words Theorem \ref{br} for $n=0$ yields exactly the Fulton-Hansen connectivity theorem.\end{remark}

\begin{lemma}\label{CT"}
Let $f\colon X\to P\times P$ be a morphism as in Theorem $\ref{br}$, with $P=\mathbb P^m(e)\times\mathbb P^n(h)$. Assume  $m\geq n\geq 1$. 
\begin{enumerate}\item[\em i)] If $\dim f(X)\geq 2m+n$ then $\dim f(X)\geq a$.
  \item[\em ii)] If $\dim f(X)>2m+ n$ then $\dim f(X)>a$.\end{enumerate}\end{lemma}

\proof Since $X_{13}\subseteq\mathbb P^m(e)\times\mathbb P^m(e)$ and $X_{24}\subseteq\mathbb P^n(h)\times\mathbb P^n(h)$, $\dim X_{13}\leq 2m$ and $\dim X_{24}\leq 2n$. It follows that 
$a=\max\{m+\dim X_{24},n+\dim X_{13}\}\leq\max\{m+2n,n+2m\}= 2m+n$. \qed

\medskip

Via Lemma \ref{CT"} we get the following Corollary of Theorem \ref{br}:

\begin{corollary}\label{CT'}
Let $f:X\to P\times P$ be a morphism as in Theorem $\ref{br}$, with $X$  a complete irreducible variety and  $P:=\mathbb P^m(e)\times \mathbb P^n(h)$, $m\geq n\geq 1$.  If $\dim f(X)>2m+ n$ then $f^{-1}(\Delta_P)$ is connected.\end{corollary}

\proof Since $\dim X_{13}\leq 2m$, $\dim X_{24}\leq 2n$ and $m\geq n$, then
$a\leq\max\{2m+n,2n+m\}= 2m+n$, and the conclusion follows from Theorem \ref{br} and Lemma \ref{CT"}. \qed

\begin{definition}[Hironaka-Matsumura \cite{HM}, or also  \cite{Ha}, or also \cite{B}, Chapter 9]\label{oo}   Let $X$ be a complete irreducible variety over the
field $k$, and let $Y$ be a closed subvariety of $X$. Denote by $K(X)$ the field of rational functions of $X$, by $X_{/Y}$ the formal completion of $X$ along $Y$, and by $K(X_{/Y})$ the ring of formal-rational functions of $X$ along $Y$.
According to Hironaka and Matsumura \cite{HM} we say that $Y$ is $G3$ in $X$ if the canonical injective map
$\alpha_{X,Y}:K(X)\to K(X_{/Y})$ is an isomorphism of $k$-algebras. In other words, $Y$ is $G3$ in $X$ if every formal rational-function of $X$ along $Y$ extends to a rational function of $X$. We also say that $Y$ is $G2$ in $X$ if the  natural injective map $\alpha_{X,Y}\colon K(X)\to K(X_{/Y})$ makes $K(X_{/Y})$ a finite field extension of $K(X)$.

Let $f\colon X'\to X$ be a proper surjective morphism of irreducible varieties, and let $Y\subset X$
and $Y'\subset X'$ be closed subvarieties such that $f(Y')\subseteq Y$. Then one can define a canonical
map of $k$-algebras $\w f^*\colon K(X_{/Y})\to K(X'_{/Y'})$ (pull back of formal-rational functions, see \cite{HM}, or also \cite{B}, Corollary 9.8) rendering commutative the following diagram:

\begin{equation*}
\begin{CD}
K(X)@>f^*>> K(X')\\
@V\alpha_{X,Y} VV @ VV\alpha_{X',Y'}V\\
K(X_{/Y})@ >\w f^*>> K(X'_{/Y'})\\
\end{CD}
\end{equation*}\end{definition}

\begin{proposition}[Hironaka--Matsumura \cite{HM}, or also \cite{B}, Cor. 9.10]\label{(2.1.11)} Let $X$ be an irreducible algebraic variety over $k$,
and let $Y$ be a closed subvariety of $X$. Let $u\colon \w X\to X$ be the $($birational$)$
normalization of $X$. Then $K(X_{/Y})$ is a field if and only if
$u^{-1}(Y)$ is connected. \end{proposition}

\begin{theorem}[Hironaka-Matsumura \cite{HM}, or also \cite{B}, Thm. 9.11]\label{(2.1.12)} Let $f\colon X'\to X$ be
a proper surjective morphism of irreducible varieties over $k$. Then for every closed
subvariety $Y$ of $X$ there is a canonical isomorphism
$$K({X'}_{/f^{-1}(Y)})\cong[K(X')\otimes_{K(X)}K({X}_{/Y})]_0,$$
where $[A]_0$ denotes the
total ring of fractions of a commutative unitary ring $A$.\end{theorem}

\begin{corollary}\label{(2.1.14)}  Under the hypotheses of Theorem
$\ref{(2.1.12)}$, assume that $\,Y$ is $G3$ in $X$. Then $f^{-1}(Y)$ is $G3$ in
$X'$.\end{corollary}

\begin{theorem}[B\u adescu--Schneider \cite{BSch}, or also \cite{B}, Cor. 9.22]\label{(2.1.24)}  Let $(X,Y)$ be a pair consisting of a complete irreducible variety $X$ over $k$ and
a closed subvariety $Y$ of $X$. The following conditions are
equivalent:
\begin{enumerate}
\item[\em i)] For every proper surjective morphism $f\colon X'\to X$
from an irreducible variety $X'$, $f^{-1}(Y)$ is connected.
\item[\em ii)] $K({X}_{/Y})$ is a field and $K(X)$ is algebraically
closed in $K({X}_{/Y})$.
\end{enumerate}
\end{theorem}

\begin{theorem}[B\u adescu \cite{Bad}]\label{debarre} Under the notation of Theorem $\ref{br}$ let  $f\colon X\to P\times P$ be a morphism from a complete irreducible variety $X$, with
$P=\mathbb P^m\times\mathbb P^n$ a product of the ordinary projective spaces $\mathbb P^m$ and $\mathbb P^n$ over $k$ and let $\Delta_P$ be the diagonal of
$P\times P$. Assume that $\dim f(X)>m+n+1$, $\dim X_{13}>m$ and $\dim X_{24}>n$. Then $f^{-1}(\Delta_P)$ is $G3$ in $X$.\end{theorem}

\begin{theorem}[B\u adescu--Schneider \cite{BSch}, or also \cite{B}, Thm. 9.21]\label{(2.1.23)} 
Let $\zeta\in K({X}_{/Y})$ be a formal-rational function of an irreducible
variety $X$ along a closed subvariety $Y$ of $X$ such that
$K({X}_{/Y})$ is a field. Then the following two conditions are equivalent:
\begin{enumerate}\item[\em i)] $\zeta$ is algebraic over $K(X)$.
\item[\em ii)]  There is a proper surjective morphism $f\colon X'\to X$ from
an irreducible variety $X'$ and a closed
subvariety $Y'$ of $X'$ such that $f(Y')\subseteq Y$ and
$\tilde{f}^*(\zeta)\in K(X')$ $($more precisely, there exists a rational
function $t\in K(X')$ such that $\w f^*(\zeta)=\alpha_{X',Y'}(t))$.
\end{enumerate}\end{theorem}

\section{Extending formal-rational functions}
\label{sec:3}

Start with the following:

\begin{lemma}\label{CT} Under the above notation let $P=\mathbb P^m(e)\times\mathbb P^n(h)$ be the product of the weighted  projective spaces $\mathbb P^m(e)$ and $\mathbb P^n(h)$ over  $k$,  let $X$ be a closed irreducible subvariety of $P\times P$, and set $a:=\max\{m+\dim X_{24},n+\dim X_{13}\}$. 
\begin{enumerate}\item[\em i)]If $\dim X>a$ 
then $\dim X>m+n+1$, $\dim X_{13}>m$ and $\dim X_{24}>n$; 
\item[\em ii)] If $\dim X\geq a$ then $\dim X\geq m+n$, $\dim X_{13}\geq m$ and $\dim X_{24}\geq n$.\end{enumerate}\end{lemma}

\begin{proof} By the hypothesis that $\dim X>a$ we get  $\dim X>m+\dim X_{24}$ and $\dim X>n+\dim X_{13}$. Denote by $p\colon X\to X_{13}$ and $q\colon X\to X_{24}$ the two canonical (surjective) projections, and by $F_p$ and  $F_q$ the general fibers of $p$ and $q$ respectively. 

\medskip

i) By way of contradiction assume for instance   that $\dim X_{13}\leq m$. Then we get successively:
 \begin{alignat*}{2}
    \dim X_{24} &<\dim X-m   &\qquad
&\textup{(by $\dim X>m+\dim X_{24}$)}  \\
    &\leq\dim X-\dim X_{13} &    & \textup{(by  $\dim X_{13}\leq m$)} \\
    &= \dim F_p  &   &   \textup{(by the theorem on dimension of fibers)}\\
     &\leq \dim X_{24} &   &   \textup{(the restriction $q|F_p\colon F_p\to X_{24}$ is injective).} 
  \end{alignat*}
Thus the assumption that $\dim X_{13}\leq m$ leads to the contradiction that $\dim X_{24}<\dim X_{24}$. This proves that $\dim X_{13}>m$. In the same manner one proves that $\dim X_{24}>n$.
Finally, from $\dim X>n+\dim X_{13}$ and $\dim X_{13}\geq m+1$ we get  $\dim X>m+n+1$.

\medskip

ii) If instead $\dim X\geq m+\dim X_{24}$ and $\dim X\geq n+\dim X_{13}$ we may again assume, by way of contradiction,  that $\dim X_{13}<m$. Then we get successively:
 \begin{alignat*}{2}
    \dim X_{24} &\leq\dim X-m   &\qquad
&\textup{(by $\dim X\geq m+\dim X_{24}$)}  \\
    &<\dim X-\dim X_{13} &    & \textup{(by  $\dim X_{13}< m$)} \\
    &= \dim F_p  &   &   \textup{(by the theorem on dimension of fibers)}\\
     &\leq \dim X_{24} &   &   \textup{(the restriction $q|F_p\colon F_p\to X_{24}$ is injective).} 
  \end{alignat*}
Thus the assumption that $\dim X_{13}< m$ leads to the same contradiction as above. This proves that $\dim X_{13}\geq m$. In the same manner one proves that $\dim X_{24}\geq n$.
Finally, from $\dim X\geq n+\dim X_{13}$ and $\dim X_{13}\geq m$ we get  $\dim X\geq m+n$.
\qed\end{proof}

\medskip

Now we can strengthen part ii) of Theorem \ref{br}  above to get the main result of this paper:

\begin{theorem}\label{br'} Under the notation of the introduction, let  $f\colon X\to P\times P$ be a morphism from a complete irreducible variety $X$ over an algebraically closed field $k$ of arbitrary characteristic, with
$P=\mathbb P^m(e)\times\mathbb P^n(h)$, and let $\Delta_P$ be the diagonal of $P\times P$.  If  $\dim f(X)> a:=\max\{m+\dim X_{24},n+\dim X_{13}\}$ then $f^{-1}(\Delta_P)$ is $G3$ in $X$.\end{theorem}

\begin{proof} By Corollary \ref{(2.1.12)} applied to the proper surjective morphism $f\colon X\to f(X)$, it is enough to prove that $f(X)\cap\Delta_P$ if $G3$ in $f(X)$. In other words, replacing $X$ by $f(X)$ we may assume that $X$ is a closed subset of $P\times P$ of dimension  $>a$ and then we have to prove that $X\cap\Delta_P$ is $G3$ in $X$.

Let $P':=\mathbb P^m\times\mathbb P^n$ be the product of two ordinary projective spaces of dimension $m$ and $n$ respectively. Then we have the canonical finite surjective morphisms $u_m(e)\colon\mathbb P^m\to\mathbb P^m(e)$ and $u_n(h)\colon\mathbb P^n\to\mathbb P^n(h)$. It follows that the morphism 
$$u:=u_m(e)\times u_n(h)\times u_m(e)\times u_n(h)\colon P'\times P'\to P\times P,$$ 
is finite and surjective. Choose an irreducible component $X'$ of $u^{-1}(X)$ and denote by $v\colon X'\to X$ the restriction $u|X'$. Clearly, $v$ is again a finite surjective morphism, and in particular, $\dim X'=\dim X$. Then it makes sense to define the irreducible subvarieties $X'_{13}\subseteq\mathbb P^m\times\mathbb P^m$ and $X'_{24}\subseteq\mathbb P^n\times\mathbb P^n$. Since the morphisms 
$$(u_m(e)\times u_m(e))|X'_{13}\colon X'_{13}\to X_{13}\;\,\text{and} \;\,(u_n(h)\times u_n(h))|X'_{24}\colon X'_{24}\to X_{24}$$ 
are finite and surjective we infer that $\dim X'_{13}=\dim X_{13}$ and $\dim X'_{24}=\dim X_{24}$, we get  
{\small\begin{equation}\label{ineg}\dim X'>a=\max\{m+\dim X_{24},n+\dim X_{13}\}=\max\{m+\dim X'_{24},n+\dim X'_{13}\}.\end{equation}}
Then by Lemma  \ref{CT}, i), the inequality \eqref{ineg} yield the following inequalities
{\small\begin{equation}\label{ineg'} \dim X'>m+n+1, \;\,\dim X'_{13}>m \;\, \text{and}\; \dim X'_{24}>n.\end{equation}}
Now, the inequalities  \eqref{ineg'} show that the hypotheses of Theorem \ref{debarre} above  are satisfied for  the inclusion $X'\subset P'\times P'$. Therefore by Theorem \ref{debarre} it follows that $X'\cap\Delta_{P'}$ is $G3$ in $X'$. 
Now the idea is to show that, in our situation, this last fact implies that $X\cap\Delta_P$ is $G3$ in $X$ as well. 

Indeed, consider the following commutative diagram
\begin{equation}\label{abc}
\begin{CD}
K(X)@>v^*>> K(X')\\
@V\alpha_{X,X\cap\Delta_P} VV @ VV\alpha_{X',X'\cap\Delta_{P'}}V\\
K(X_{/X\cap\Delta_P})@ >\w v^*>> K(X'_{/X'\cap\Delta_{P'}})\\
\end{CD}
\end{equation}
in which the second vertical map is an isomorphism (because $X'\cap\Delta_{P'}$ is $G3$ in $X'$) and the first horizontal map yields a finite field extension (because the morphism $v$ is finite). In particular, via the injective map 
$\alpha_{X',X'\cap\Delta_{P'}}\circ v^*$, $K(X'_{/X'\cap\Delta_{P'}})$ becomes a finite field extension of $K(X)$.

On the other hand, we claim that the ring $K(X_{/X\cap\Delta_P})$ of formal-rational functions of $X$ along $X\cap\Delta_P$ is actually a field. Indeed by Proposition \ref{(2.1.11)} we have to check that if $f\colon\tilde{X}\to X$  is the
birational normalization of $X$, then $f^{-1}(X\cap\Delta_P)$ is connected. But the connectivity of $f^{-1}(X\cap\Delta_P)$ follows from Theorem \ref{br}. So 
$K(X_{/X\cap\Delta_P})$ field, and hence it can be identified with a subfield of $K(X'_{/X'\cap\Delta_{P'}})$ which contains $K(X)$. By the commutativity of 
diagram \eqref{abc} we get $$\w v^*\circ\alpha_{X,X\cap\Delta_P}=\alpha_{X',X'\cap\Delta_{P'}}\circ v^*,$$ so that the field extension $K(X_{/X\cap\Delta_P})|K(X)$ becomes a field subextension of the finite field extension $K(X'_{/X'\cap\Delta_{P'}})|K(X)$. It follows that the field extension  $K(X_{/X\cap\Delta_P})|\newline K(X)$ is finite, i.e.  $X\cap\Delta_P$ is $G2$ in $X$. 

It remains to see that the map $\alpha_{X,X\cap\Delta_P}$ is an isomorphism. Under our assumption that $\dim X>a=\max\{m+\dim X_{24},n+\dim X_{13}\}$, we can apply Theorem \ref{br}, ii)
to get that the condition i) of Theorem \ref{(2.1.24)} is satisfied for the pair $(X,X\cap\Delta_P)$. Then by Theorem \ref{(2.1.24)} above, this condition is 
equivalent to saying that the subfield $K(X)$ is algebraically closed in $K(X_{/X\cap\Delta_P})$.  Recalling also that the field extension 
$K(X_{/X\cap\Delta_P})|K(X)$ is finite (and hence algebraic) we get that the map $\alpha_{X,X\cap\Delta_P}$ is an isomorphism, i.e.  $X\cap\Delta_P$ is $G3$ in $X$.\qed\end{proof}

\begin{corollary}\label{small} Let  $f\colon X\to P\times P$ be a morphism from a complete irreducible variety $X$, with  $P=\mathbb P^m(e)\times \mathbb P^n(h)$,  such that $m\geq n\geq 1$ and $\codim_{P\times P}f(X)< n$. Then $f^{-1}(\Delta)$ is $G3$ in $X$.
\end{corollary}

\begin{proof} This follows from Theorem \ref{br'} and Corollary \ref{CT'}.\qed\end{proof}

\begin{remark} { In the special case when $P=\mathbb P^m\times\mathbb P^n$ is a product of two ordinary projective spaces over the  field $\mathbb C$ of complex numbers, Corollary 1  also follows
(via Corollary \ref{(2.1.14)})  from an old general result of Faltings (see \cite{F}, Satz 8, page 161) proved in the case when $P$ is a complex projective rational homogeneous space. In general $P=\mathbb P^m(e)\times\mathbb P^n(h)$ is  singular, so that Corollary 1 (to our best knowledge) is new.}\end{remark}

\begin{corollary}\label{small'} Let $X$ and $Y$ be two closed  irreducible subvarieties of $P=\mathbb P^m(e)\times\mathbb P^n(h)$  such that  $m\geq n\geq 1$ and $\dim X+\dim Y>
2m+n$. Then $X\cap Y$ is $G3$ in $X$ and in $Y$.\end{corollary}

\begin{proof}   Let $p_1\colon X\times Y\to X$ be the first projection of $X\times Y$.  Then we get the commutative diagram:
$$\begin{CD}
(X\times Y)\cap \Delta_P@>\subset>> X\times Y\\
@V\cong VV @ VVp_1V\\
X\cap Y@ >\subset>> X\\
\end{CD}$$
yields the  following commutative diagram
\begin{equation}\label{sd}\begin{CD}
K(X)@>p_1^*>> K(X\times Y)\\
@V\alpha_{X,X\cap Y} VV @ VV\alpha_{X\times Y,(X\times Y)\cap\Delta_P} V\\
K(X_{/X\cap Y})@ >\w p_1^*>> K(X\times Y_{/(X\times Y)\cap\Delta_P})\\
\end{CD}\end{equation}
 As $\dim X\times Y=\dim X+\dim Y>2m+n$, by Corollary \ref{small} we get that $(X\times Y)\cap\Delta_P$ is $G3$ in $X\times Y$, so that the second vertical arrow of diagram 
\eqref{sd} is an isomorphism. 

On the other hand, we claim that for every proper surjective morphism $f\colon Z\to X$, $f^{-1}(X\cap Y)$ is connected. Indeed, since the morphism $f\times\id_Y\colon Z\times Y\to X\times Y$ is proper and surjective (because $f\colon Z\to X$ is so), and since $(X\times Y)\cap\Delta_P$ is $G3$ in $X\times Y$, by Corollary \ref{CT'},
$(f\times\id_Y)^{-1}((X\times Y)\cap\Delta_P)$ is $G3$ in $Z\times Y$. It follows in particular that $(f\times\id_Y)^{-1}((X\times Y)\cap\Delta_P)$ is connected. As $(f\times\id_Y)^{-1}((X\times Y)\cap\Delta_P)$ 
is biregularly isomorphic to $f^{-1}(X\cap Y)$, the claim is proved. 

The claim implies the following two things: 

\begin{enumerate}\item[ i)] $K(X_{/X\cap Y})$ is a field (by Proposition \ref{(2.1.11)}), and 
  \item[ ii)] $K(X)$ is algebraically closed in $K(X_{/X\cap Y})$ (by Theorem \ref{(2.1.24)}). \end{enumerate}

Now we can easily prove that $X\cap Y$ is $G3$ in $X$. Indeed, if not, there would exist a formal-rational function $\zeta\in K(X_{/X\cap Y})$ such that $\zeta\not\in K(X)$. Then by diagram \eqref{sd} (with the second vertical arrow isomorphism) and by Theorem \ref{(2.1.23)} it would follow that in the field extension $K(X_{/X\cap Y})|K(X)$ the function $\zeta\in K(X_{/X\cap Y})$ would be an algebraic element over $K(X)$ non belonging to $K(X)$, and this would contradict i) and ii) above.

Similarly one proves that $X\cap Y$ is $G3$ in $Y$.\qed\end{proof}

\begin{remark} Corollary \ref{small'} extends to the case when $P=\mathbb P^m(e)\times\mathbb P^n(h)$ an old result of Faltings \cite{F2} (see Corollary 3, page 102) regarding the case when $P=\mathbb P^n$. Our proof (based on  global arguments) is different from Faltings' proof (which uses local methods).
\end{remark}

\begin{corollary}\label{final} Under the hypotheses of Corollary $\ref{small'}$ assume that $k$ is the field of complex numbers. Then every meromorphic function defined on a complex connected open neighborhood $U$ of $X\cap Y$ in $X$ extends to a rational function in $X$.\end{corollary}

\begin{proof} Let $\mathscr M(U)$ denote the field of meromorphic functions on $U$. By \cite{B}, page 117, $K(X)\subseteq\mathscr M(U)\subseteq K(X_{X\cap Y})$, and then apply Corollary \ref{small'}.   \qed\end{proof}

\end{document}